\documentclass[12pt]{article}
\usepackage[T1]{fontenc}
\usepackage{amsmath}
\usepackage{amsfonts}
\usepackage{verbatim}
\usepackage{bbm}
\usepackage{amsthm, amssymb}
\newtheorem{theorem}{Theorem}
\newtheorem{lemma}[theorem]{Lemma}
\newtheorem{proposition}[theorem]{Proposition}
\newtheorem{corollary}[theorem]{Corollary}

\theoremstyle{definition}
\newtheorem{definition}[theorem]{Definition}

\theoremstyle{remark}
\newtheorem{remark}[theorem]{Remark}

\theoremstyle{remark}

 \numberwithin{equation}{section}
\numberwithin{theorem}{section}

\date{}

\title{\bf Concentration and correlation inequalities with size biased couplings}

\newcommand{\bea}{\begin{eqnarray}}
\newcommand{\ena}{\end{eqnarray}}
\newcommand{\beas}{\begin{eqnarray*}}
\newcommand{\enas}{\end{eqnarray*}}
\newcommand{\beq}{\begin{equation}}
\newcommand{\enq}{\end{equation}}

\newcommand{\ignore}[1]{}

\newcommand{\qmq}[1]{\quad\mbox{#1}\quad}

\newcommand{\bs}{\boldsymbol}
\newcommand{\bb}{\mathbb}

\title{\bf Multivariate Concentration Inequalities with Size Biased Couplings}
\author{Subhankar Ghosh and \"{U}m\.{i}t I\c{s}lak}

\date{}

\begin{document}
 \maketitle

\vspace{0.25in}

\begin{abstract}
Let $\mathbf{W}=(W_1,W_2,...,W_k)$ be a random vector with
nonnegative coordinates having nonzero and finite variances. We
prove concentration inequalities for $\mathbf{W}$ using size biased
couplings that generalize the previous univariate results. Two
applications on local dependence and counting patterns are provided.
\end{abstract}

\section{Introduction}

The purpose of this paper is to show how size biased couplings can
be used to obtain multivariate concentration inequalities in
dependent settings. For a given random vector
$\mathbf{W}=(W_1,W_2,...,W_k)$ with nonnegative coordinates having
finite, nonzero expectations $\mu_i=\mathbb{E}[W_i]$, another vector
$\mathbf{W}^i=(W_1^i,W_2^i,...,W_k^i)$ is said to have $\mathbf{W}$
size biased distribution in direction $i$ if
\begin{equation}\label{sizebiasdefn1}
    \mathbb{E}[W_i f(\mathbf{W})] = \mu_i \mathbb{E}[f(\mathbf{W}^i)]
\end{equation}
for all functions for which these expectations exist. For a
univariate nonnegative random variable $W$ with mean
$\mu=\mathbb{E}[W] \in (0,\infty)$, this simplifies to
$\mathbb{E}[Wf(W)] = \mu \mathbb{E}[f(W^s)]$ and we say that $W^s$
has $W$ size biased distribution.
 We refer \cite{arratia} and \cite{brown} for two excellent expository papers
 on several aspects of
 size biasing.

In a broad sense, concentration inequalities quantify the fact that
a function of a large number of random variables, with certain
smoothness conditions,  tends to concentrate its values in a
relatively narrow range.
 There is a tremendous
literature on inequalities for functions of independent random
variables due to their importance in several fields. See, for
example \cite{chung} and \cite{mcdiarmid} for wonderful surveys,
\cite{boucheron} and \cite{ledoux} for book length treatments of the
subject. Our approach here will be on the use of couplings from
Stein's method which is a technique introduced by Charles Stein in
\cite{stein} that is used for obtaining error bounds in
distributional approximations. The strength of the method comes from
the fact that it can be also used for functions of dependent random
variables and various coupling constructions are used for such
problems. Sourav Chatterjee in \cite{chat2} used one important
coupling from Stein's method, exchangeable pairs, to show the
concentration of several interesting statistics. Later, in
\cite{ghosh1}, Subhankar Ghosh and Larry Goldstein were able to
obtain similar bounds with size biased couplings. Our results in
Section 2 will provide  multivariate analogues of Ghosh and
Goldstein's results and will also yield a partial improvement in
their lower tail inequality.

The paper is organized as follows. In Section \ref{mainsection}, we
state the multivariate concentration bound with  size biased
couplings, consider its univariate corollary and discuss briefly the
construction of size biased couplings. Proofs of the results are
given in Section \ref{proofs} and we provide two applications, one
on local dependence and the other one on counting patterns in random
permutations, in Section \ref{appssection}.

\section{Main result}\label{mainsection}

We start by fixing some notations. Throughout this paper, for two
vectors $\mathbf{x}, \mathbf{y} \in \mathbb{R}^k$, we will write
$$\frac{\mathbf{x}}{\mathbf{y}}= \left(\frac{x_1}{y_1},\frac{x_2}{y_2},...,\frac{x_k}{y_k}
\right)$$ for convenience. Also, we define the partial ordering
$\succeq$ on $\mathbb{R}^k$ by
$$\mathbf{x} \succeq \mathbf{y} \Leftrightarrow x_i \geq y_i, \quad \text{for} \; i=1,2,...,k.$$
Accordingly, the order $\preceq$ is defined by $\mathbf{x} \preceq
\mathbf{y} \Leftrightarrow \mathbf{y} \succeq \mathbf{x}$, and the
definitions for $\prec$ and $\succ$ are similar. Finally, for
$\bs{\theta} \in \bb{R}^k$, $\bs{\theta}^t$ will stand for the
transpose of $\bs{\theta}$ and  $\|\bs{\theta}\|_2$ is the $l^2$
norm of $\bs{\theta}$. Now, we are ready to state our main result.

\begin{theorem}\label{sbmainthm}
Let $\mathbf{W}=(W_1,W_2,...,W_k)$ be a random vector where $W_i$ is
nonnegative with mean $\mu_i>0$ and variance $\sigma_i^2\in
(0,\infty)$ for each $i=1,2,...,k$, and suppose that the moment
generating function of $\mathbf{W}$ exists everywhere. Assuming that
we can find couplings $\{\mathbf{W}^i\}_{i=1}^k$ of $\mathbf{W}$,
with $\mathbf{W}^i$ having $\mathbf{W}$ size biased distribution in
direction $i$ and satisfying $\|\mathbf{W}^i-\mathbf{W}\|_2\leq K$
for some constant $K>0$, we have
 \bea\label{sblefttail}\mathbb{P} \left(\frac{\bf{W}-
\bs{\mu}}{\bs{\sigma} } \preceq -\mathbf{t} \right) \leq \exp
\left(- \frac{\|\mathbf{t}\|_2^2}{2 K_1} \right)\ena and
\bea\label{sbrighttail}\mathbb{P} \left(\frac{\mathbf{W}-
\bs{\mu}}{\bs{\sigma} } \succeq \mathbf{t} \right) \leq \exp \left(-
\frac{\|\mathbf{t}\|_2^2}{2(K_1  + K_2 \|\mathbf{t}\|_2)
}\right)\ena
 \noindent for any $\mathbf{t}\succeq
\mathbf{0}$ where $K_1= \frac{2 K}{\sigma_{(1)}}
\left\|\frac{\bs{\mu}}{\bs{\sigma}} \right\|_2$, $K_2= \frac{K}{2
\sigma_{(1)}}$ with $\sigma_{(1)}=\min_{i=1,2,...,k} \sigma_i$,
$\bs{\mu}=(\mu_1,\mu_2,...,\mu_k)$ and
$\bs{\sigma}=(\sigma_1,\sigma_2,...,\sigma_k).$
\end{theorem}

Proof of Theorem \ref{sbmainthm} will be given in Section
\ref{proofs}. Here we note that the assumption on moment generating
function (mgf) can be relaxed to $\mathbb{E}[e^{\bs{\theta}^t
\mathbf{W}}]< \infty$ for $\|\bs{\theta}\|_2 \leq 2/K$, as can be
checked easily from the proof. As a more general remark, Arratia and
Baxendale showed recently for the univariate case that the existence
of a bounded coupling for $W$ assures the existence of the mgf
everywhere. See \cite{arratiabaxendale} for details. Although a
similar result can be given in a multivariate setting, we skip this
for now as in applications the underlying random variables are
almost always finite (so that mgf exists everywhere).

Noting that the $k=1$ case in Theorem \ref{sbmainthm} reduces to
standard size biasing and replacing $t$ by $t/ \sigma$, we arrive at
the following univariate corollary.

\begin{corollary}\label{univariatecase}
Let $W$ be a nonnegative random variable with finite and nonzero
mean, and assume that the moment generating function of $W$ exists
everywhere. If there exists a size biased coupling $W^s$ of
 $W$ satisfying $|W^s-W| \leq K$ for some
$K>0$, then for any $t \geq 0$, we have
\begin{equation}\label{ltcor}
\mathbb{P}(W-\mu \leq -t ) \leq \exp\left(- \frac{t^2}{4K \mu}
\right) \quad \text{and} \quad \mathbb{P}(W -\mu \geq t) \leq \exp
\left( -\frac{t^2}{4K \mu + Kt} \right).
\end{equation}
\end{corollary}

\begin{remark}
For the one dimensional case, the lower tail inequality in
(\ref{ltcor}) improves Ghosh and Goldstein's corresponding result
(namely, inequality (1) in \cite{ghosh1}) by removing the
monotonicity condition. 
However, in both tails the constants are slightly worse than the
ones in their theorem, but this is not too surprising  as our main
result is proven for a multivariate version. We note that this
monotonicity condition is also discussed in two recent papers,
\cite{arratiabaxendale} and \cite{BaGoU13}, where they prove that it
is indeed possible to remove the monotonicity condition while
keeping the bound exactly the same as in \cite{ghosh1}.
\end{remark}

\begin{remark}
For the upper tail in univariate case, there has been a recent
improvement in \cite{arratiabaxendale} where the authors show that
it is indeed possible to obtain a tail behavior of order $\exp(-c t
\log t)$ under bounded size biased coupling assumption. In
particular, this result reveals the upper tail inequality given in
\cite{ghosh1} as a corollary. However, we were not able to obtain a
similar bound for the multivariate case yet, and this will be one
direction to follow in a subsequent work.
\end{remark}

In the rest of this section we briefly review the discussion in
\cite{ghosh2} which gives a procedure to size bias a collection of
nonnegative random variables in a given direction. More on
construction of size biased couplings can be found in
\cite{goldstein1}. Now, as mentioned in the introduction, for a
random vector $\mathbf{W}=(W_1,W_2,...,W_k)$ with nonnegative
coordinates, a random variable $\mathbf{W}^i$ is said to have
$\mathbf{W}$ size bias distribution in direction $i$ if
$\mathbb{E}[W_i f(\mathbf{W})] = \mu_i \mathbb{E}[f(\mathbf{W}^i)]$
for all functions for which these expectations exist. It is well
known that the definition just given is equivalent to the following
one.

\begin{definition}
Let $\mathbf{W}=(W_1,W_2,...,W_k)$ be a  random vector where $W_j$'s
have finite, nonzero expectations $\mu_j = \mathbb{E}[W_j]$ and
joint distribution $dF(\mathbf{x})$. For $i \in \{1,2,...,k\}$, we
 say that $\mathbf{W}^i = (W_1^i,W_2^i,...,W_k^i)$ has the $\mathbf{W}$ size bias distribution
in direction $i$ if $\mathbf{W}^i$ has joint distribution
\begin{equation}\label{sizebiasdefn2}
dF^{i} (\mathbf{x}) = \frac{x_i dF(\mathbf{x})}{\mu_i}.
\end{equation}
\end{definition}

Note that in univariate case, (\ref{sizebiasdefn2}) reduces to
$dF^{*} (x) = x dF(x)/\mu$ which explains the name, \emph{size
biased distribution}. Also this latter definition gives insight for
a way to construct size biased random variables. Following
\cite{ghosh2}, by the factorization of $dF(\mathbf{x})$, we have
\begin{eqnarray*}
 dF^i(\mathbf{x}) = \frac{x_i dF(\mathbf{x})}{\mu_i}  &=&\mathbb{P}(\mathbf{W} \in d\mathbf{x} | W_i = x)  \frac{x_i  \mathbb{P}(W_i \in dx)}{\mu_i}
    = \mathbb{P}(\mathbf{W} \in d\mathbf{x} |W_i =x) \mathbb{P}(W_i^i \in    dx).
\end{eqnarray*}

\noindent where $W_i^i$ has $W_i$ size biased distribution. Hence,
to generate $\mathbf{W}^i$ with distribution $dF^i$, first generate
a variable $W_i^i$ with $W_i$ size bias distribution. Then, when
$W_i^i = x$, we generate the remaining variables according to their
original conditional distribution given that $i^{th}$ coordinate
takes on the value $x.$

As an example, the construction just described combined with Theorem
\ref{sbmainthm} can be used   to prove concentration bounds for
random vectors with independent coordinates. To see this in the
simplest possible case, let $\mathbf{W}=(W_1,...,W_k)$ be a random
vector where $W_i$'s are
 nonnegative, independent and identically distributed  random
variables with  $W_i \leq K$ a.s. for some $K>0$, and assume that $0
< \sigma^2 = Var(W_1) < \infty$. To obtain $\mathbf{W}^i$, we let
$W_i^i$ be on the same space with $W_i$ size biased distribution and
also set $W_j^i = W_j$ for $j \neq i$. Since coordinates of
$\mathbf{W}$ are independent, $\mathbf{W}^i=(W_1^i,W_2^i,...,W_k^i)$
has $\mathbf{W}$ size biased distribution in direction $i.$ Also
noting that $W_i^i \leq K$ as support of $W_i^i$ is a subset of the
support of $W_i$, we obtain $\|\mathbf{W}^i - \mathbf{W}\|_2 \leq K$
a.s. and using Theorem \ref{sbmainthm}, one can conclude that the
lower tail inequality $$\mathbb{P} \left(\frac{\mathbf{W}-
\bs{\mu}}{\bs{\sigma} } \preceq -\mathbf{t} \right) \leq \exp
\left(-\frac{\sigma^2 \|\mathbf{t}\|_2^2}{4 K \sqrt{k} \mu} \right)
$$ and the upper tail inequality $$ \mathbb{P} \left(\frac{\mathbf{W}- \bs{\mu}}{\bs{\sigma}} \succeq
\mathbf{t} \right) \leq \exp \left(- \frac{\|\mathbf{t}\|_2^2}{ 4 K
\sqrt{k} \mu / \sigma^2 + K \|\mathbf{t}\|_2 /\sigma} \right)$$ hold
for all $\mathbf{t} \succeq \mathbf{0}.$

\section{Proofs}\label{proofs}

Before we begin the proofs, we note the following inequality
\bea\label{convexexp2} |e^{y} - e^{x}| \leq |y-x| \left(\frac{e^{y}
+ e^{x}}{2}\right) \ena which follows from the following observation
\beas\label{convexexp} \frac{e^y - e^x}{y-x} = \int_0^1 e^{t y
+(1-t) x} dt \leq \int_0^1 (t e^y + (1-t) e^x) dt= \frac{e^y +
e^x}{2} \qmq{for all $x \neq y$.} \enas

\textbf{Proof of Theorem \ref{sbmainthm}.} We first prove the upper
tail inequality. Let $\bs{\theta} \succeq \mathbf{0} =
(0,0,...,0)\in \mathbb{R}^k$ with $\|\bs{\theta}\|_2 < 2 /K$. Note
that an application of (\ref{convexexp2}) and Cauchy-Schwarz
inequality gives for any $i=1,...,k$
\begin{eqnarray*}     
                    \mathbb{E}[e^{\bs{\theta}^t \mathbf{W}^i}] - \mathbb{E}[e^{\bs{\theta}^t \mathbf{W}}] \leq |\mathbb{E}[e^{\bs{\theta}^t \mathbf{W}^i}] - \mathbb{E}[e^{\bs{\theta}^t
                                                                 \mathbf{W}}]|
                                                                  &\leq&  \mathbb{E} \left[\frac{|\bs{\theta}^t (\mathbf{W}^i - \mathbf{W})| (e^{\bs{\theta}^t \mathbf{W}^i}+ e^{\bs{\theta}^t \mathbf{W}})}{2} \right]\\
                                                                  &\leq& \mathbb{E}\left[\frac{\|\bs{\theta}\|_2 \|\mathbf{W}^i -\mathbf{W}\|_2 (e^{\bs{\theta}^t \mathbf{W}^i} + e^{\bs{\theta}^t \mathbf{W}})}{2} \right] \\
                                                                  &\leq&
                                                                  \frac{K \|\bs{\theta}\|_2}{2} \mathbb{E}[e^{\bs{\theta}^t \mathbf{W}^i} + e^{\bs{\theta}^t
                                                                  \mathbf{W}}].
                                                               \end{eqnarray*}

Changing sides, since $\|\bs{\theta}\|_2 < 2/K,$ we obtain
\begin{equation}\label{utest1}
\mathbb{E}[e^{\bs{\theta}^t \mathbf{W}^i}] \leq \frac{1+ \frac{K
\|\bs{\theta}\|_2}{2}}{1- \frac{K \|\bs{\theta}\|_2}{2}}
\mathbb{E}[e^{\bs{\theta}^t \mathbf{W}}].
\end{equation}

\noindent Letting $m(\bs{\theta})=\mathbb{E}[e^{\bs{\theta}^t
\mathbf{W}}]$, using (\ref{utest1}) and the size bias relation in
(\ref{sizebiasdefn1}) we have
\begin{equation}\label{mestim}
    \frac{\partial m(\bs{\theta})}{\partial \theta_i} = \mathbb{E}[W_i e^{\bs{\theta}^t
    \mathbf{W}}] = \mu_i \mathbb{E}[e^{\bs{\theta}^t \mathbf{W}^i}]
    \leq \mu_i \frac{1+ \frac{K \|\bs{\theta}\|_2}{2}}{1- \frac{K
\|\bs{\theta}\|_2}{2}} \mathbb{E}[e^{\bs{\theta}^t \mathbf{W}}] =
\mu_i \frac{2+ K \|\bs{\theta}\|_2}{2- K \|\bs{\theta}\|_2}
m(\bs{\theta}).
\end{equation}

Now, letting $M(\bs{\theta}) = \mathbb{E}\left[\exp
\left(\bs{\theta}^t
\left(\frac{\mathbf{W}-\bs{\mu}}{\bs{\sigma}}\right)\right)\right],$
 observe that we have $ M(\bs{\theta}) =
m\left(\frac{\bs{\theta}}{\bs{\sigma}}\right)\exp
\left(-\bs{\theta}^t \frac{\bs{\mu}}{\bs{\sigma}}\right). $
 Hence denoting $$\partial_i m(\bs{\beta}) = \frac{\partial m(\bs{\theta})}{\partial \theta_i} \big|_{\bs{\theta} = \bs{\beta}} \; ,$$
 we obtain for $\|\bs{\theta}/\bs{\sigma}\|_2 < 2/K$,
\begin{eqnarray*}
  \frac{\partial M(\bs{\theta})}{\partial \theta_i} &=& \frac{1}{\sigma_i}
  \partial_i m \left(\frac{\bs{\theta}}{\bs{\sigma}} \right)\exp\left(-\theta^t
\frac{\bs{\mu}}{\bs{\sigma}}\right) - \frac{\mu_i}{\sigma_i} m
\left(\frac{\bs{\theta}}{\bs{\sigma}} \right)\exp
\left(-\bs{\theta}^t \frac{\bs{\mu}}{
\bs{\sigma}}\right)\\
    &\leq& \frac{\mu_i}{\sigma_i} \left(\frac{2+K \|\bs{\theta} / \bs{\sigma}\|_2}{2-K \|\bs{\theta} / \bs{\sigma}\|_2} \right) m \left(\frac{\bs{\theta}}{\bs{\sigma}}\right) \exp \left(- \bs{\theta}^t \frac{\bs{\mu}}{\bs{\sigma}} \right) - \frac{\mu_i}{\sigma_i} m \left(\frac{\bs{\theta}}{\bs{\sigma}} \right) \exp \left(- \bs{\theta}^t \frac{\bs{\mu}}{\bs{\sigma}} \right)\\
   &=& \frac{\mu_i}{\sigma_i} M(\bs{\theta}) \left(\frac{2 + K \|\bs{\theta} / \bs{\sigma}\|_2}{ 2 - K \|\bs{\theta}/ \bs{\sigma}\|_2} -1
   \right) \\
   &=& \frac{\mu_i}{\sigma_i} M(\bs{\theta}) \left(\frac{2K \|\bs{\theta}/ \bs{\sigma}\|_2}{2 - K \|\bs{\theta} /
   \bs{\sigma}\|_2}\right).
\end{eqnarray*}

This in particular gives for $\|\bs{\theta} / \bs{\sigma}\|_2 < 2
/K$, $$\frac{\partial \log M(\bs{\theta})}{\partial \theta_i} \leq
\frac{\mu_i}{\sigma_i} \frac{2K \|\bs{\theta}/ \bs{\sigma}\|_2}{2 -
K \|\bs{\theta} /
   \bs{\sigma}\|_2}.$$
Now, using the mean value theorem, for all $\mathbf{0} \preceq
\bs{\theta} \in \mathbb{R}^k$ with $\| \bs{\theta} / \bs{\sigma}
\|_2 < 2/K,$
$$\log (M(\bs{\theta})) = \nabla \log (M(\mathbf{z})) \cdot \bs{\theta}, $$ for some $0 \preceq \mathbf{z} \preceq
\bs{\theta}.$ Noting that  $\|\mathbf{z} / \bs{\sigma}\|_2 \leq
\|\bs{\theta} / \bs{\sigma}\|_2 < 2/K$ and using Cauchy-Schwarz
inequality, we obtain
\begin{eqnarray}\label{logboundright}
  \log M(\theta) = \nabla \log M(\mathbf{z}) \cdot \bs{\theta} &\leq& \sum_{i=1}^k  \frac{2K \|\mathbf{z}/ \bs{\sigma}\|_2}{2 - K \|\mathbf{z} /
   \bs{\sigma}\|_2} \frac{\mu_i}{\sigma_i} \theta_i
    \leq \frac{2K \|\bs{\theta}/ \bs{\sigma}\|_2}{2 - K \|\bs{\theta} /
    \bs{\sigma}\|_2} \left\|\frac{\bs{\mu}}{\bs{\sigma}} \right\|_2 \|\bs{\theta}\|_2
\end{eqnarray}
Next we observe that $$\|\bs{\theta}\|_2 < \frac{1}{K_2} \implies
\left\|\frac{\bs{\theta}}{\bs{\sigma}}\right\|_2 < \frac{2}{K}.$$
Thus if $\|\bs{\theta}\|_2 < 1 / K_2,$ (\ref{logboundright}) yields
$$\log M(\bs{\theta}) \leq \left\|\frac{\bs{\mu}}{\bs{\sigma}}
\right\|_2 \frac{2 K \|\bs{\theta}\|_2^2/\sigma_{(1)}}{(2 - K
\|\bs{\theta}\|_2 / \sigma_{(1)})} = \frac{K_1
\|\bs{\theta}\|_2^2}{2 (1-K_2\|\bs{\theta}\|_2)}.$$ Hence if
$\mathbf{t} \succeq 0$ and $\|\bs{\theta}\|_2 < 1/K_2$, an
application of Markov's inequality yields
\begin{eqnarray*}
  \mathbb{P}\left(\frac{\mathbf{W}-\bs{\mu}}{\bs{\sigma}} \succeq \mathbf{t} \right) \leq \mathbb{P} \left(\bs{\theta}^t \left(\frac{\mathbf{W}-\bs{\mu}}{\bs{\sigma}} \right)
\geq \bs{\theta}^t \mathbf{t} \right) &\leq& \exp
\left(-\bs{\theta}^t
\mathbf{t}\right) M(\bs{\theta})  \\
   &\leq& \exp \left(-\bs{\theta}^t \mathbf{t} +
\frac{K_1 \|\bs{\theta}\|_2^2}{2 (1 - K_2 \|\bs{\theta}\|_2) }
\right)
\end{eqnarray*}

Using $\bs{\theta} =\frac{\mathbf{t}}{K_1+K_2 \|\mathbf{t}\|_2}
\succeq \mathbf{0}$, and noting that $\|\bs{\theta}\|_2<1/K_2$, we
finish the proof of the upper tail inequality.

\vspace{0.1in}

Next we prove the lower tail bound given in (\ref{sblefttail}).
Letting $\bs{\theta} \preceq \mathbf{0}$ and using the size bias
relation given in (\ref{sizebiasdefn1}), we have
$$\frac{\partial m (\bs{\theta})}{\partial \theta_i} = \mathbb{E}[W_i
e^{\bs{\theta}^t \mathbf{W}}] = \mu_i
\mathbb{E}[e^{\mathbf{\bs{\theta}}^t \mathbf{W}^i}]=\mu_i
\mathbb{E}[e^{\mathbf{\bs{\theta}}^t (\mathbf{W}^i - \mathbf{W})}
e^{\bs{\theta}^t \mathbf{W}}].$$ Using the inequality $e^x \geq 1 +
x$, this yields
\begin{equation}\label{ltpartest}
\frac{\partial m}{\partial \theta_i} \geq \mu_i \mathbb{E}[(1 +
\bs{\theta}^t (\mathbf{W}^i-\mathbf{W}))e^{\bs{\theta}^t
\mathbf{W}}].
\end{equation}
 By Cauchy-Schwarz inequality and that $\|\mathbf{W}^i - \mathbf{W}\|_2 \leq K$, we have $$|\bs{\theta}^t
(\mathbf{W}^i-\mathbf{W})| \leq \|\bs{\theta}\|_2
\|\mathbf{W}^i-\mathbf{W}\|_2 \leq K \|\bs{\theta}\|_2$$ which in
particular gives $\bs{\theta}^t (\mathbf{W}^i-\mathbf{W}) \geq -K
\|\bs{\theta}\|_2$. Combining this observation with
(\ref{ltpartest}), we arrive at

\begin{equation}\label{ltpartest2}
\frac{\partial m}{\partial \theta_i} \geq \mu_i \mathbb{E}[(1-K
\|\bs{\theta}\|_2) e^{\bs{\theta}^t \mathbf{W}}]= \mu_i (1-K
\|\bs{\theta}\|_2) m(\bs{\theta}).
\end{equation}
Now, keeping the notations as in the upper tail case and using the
estimate in (\ref{ltpartest2}), we get
\begin{eqnarray*}
  \frac{\partial M }{\partial \theta_i} &=& \frac{1}{\sigma_i}
  \partial_i m \left(\frac{\bs{\theta}}{\bs{\sigma}} \right)\exp\left(-\theta^t
\frac{\bs{\mu}}{\bs{\sigma}}\right) - \frac{\mu_i}{\sigma_i} m
\left(\frac{\bs{\theta}}{\bs{\sigma}} \right)\exp
\left(-\bs{\theta}^t \frac{\bs{\mu}}{
\bs{\sigma}}\right)\\
&=& \frac{1}{\sigma_i} \exp \left(-\bs{\theta}^t \frac{\bs{\mu}}{
\bs{\sigma}}\right) \left(\partial_i
m\left(\frac{\bs{\theta}}{\bs{\sigma}} \right)-\mu_i
m\left(\frac{\bs{\theta}}{\bs{\sigma}} \right) \right) \\
&\geq& \frac{1}{\sigma_i} \exp \left(-\bs{\theta}^t \frac{\bs{\mu}}{
\bs{\sigma}}\right) \left\{\mu_i \left(1-K
\left\|\frac{\bs{\theta}}{\bs{\sigma}}\right\|_2 \right)
m\left(\frac{\bs{\theta}}{\bs{\sigma}} \right) - \mu_i
m\left(\frac{\bs{\theta}}{\bs{\sigma}} \right) \right\}
\end{eqnarray*}
Manipulating the terms in the lower bound, this yields
\begin{eqnarray*}\frac{\partial M }{\partial \theta_i} &=&
\frac{1}{\sigma_i} \exp \left(-\bs{\theta}^t \frac{\bs{\mu}}{
\bs{\sigma}}\right) \left(-\mu_i K
\left\|\frac{\bs{\theta}}{\bs{\sigma}} \right\|_2 m
\left(\frac{\bs{\theta}}{\bs{\sigma}} \right) \right) \\
&=& - \frac{\mu_i}{\sigma_i} K
\left\|\frac{\bs{\theta}}{\bs{\sigma}} \right\|_2 M(\bs{\theta})
\\
&\geq&  -\frac{\mu_i}{\sigma_i}
\frac{K}{\sigma_{(1)}}\|\bs{\theta}\|_2 M(\bs{\theta}).
  \end{eqnarray*}
Now, using the mean value theorem, for $\bs{\theta} \preceq
  \mathbf{0}$, one can find $\bs{\theta} \preceq \mathbf{z} \preceq \mathbf{0}$
  such  that $$\log M(\bs{\theta})= \nabla \log M(\mathbf{z}) \cdot
  \bs{\theta}.$$
Hence for a given $\bs{\theta} \preceq \mathbf{0}$, we have
\begin{eqnarray}\label{logMest}
            \log M(\bs{\theta})= \nabla \log M(\mathbf{z}) \cdot \bs{\theta}
            &\leq& \sum_{i=1}^k \left(\left(\frac{-K \mu_i \|\bs{\theta}\|_2}{\sigma_{(1)} \sigma_i} \right)  \theta_i\right)
\end{eqnarray}
where we used that $\theta_i\leq 0$ for each $i$ for the
inequalities. Now, using (\ref{logMest}) and an application of
Cauchy-Schwarz inequality gives
\begin{eqnarray*}
          \log M(\bs{\theta})&\leq& \frac{K \|\bs{\theta}\|_2}{\sigma_{(1)}}
          \sum_{i=1}^k \frac{\mu_i |\theta_i|}{\sigma_i}
          \leq  \frac{K}{\sigma_{(1)}} \left\|\frac{\bs{\mu}}{\bs{\sigma}}\right\|_2
          \|\bs{\theta}\|_2^2.
\end{eqnarray*}
which after exponentiation yields $$M(\bs{\theta}) \leq
\exp\left(\frac{K}{\sigma_{(1)}}
\left\|\frac{\bs{\mu}}{\bs{\sigma}}\right\|_2
          \|\bs{\theta}\|_2^2\right).$$ Combining this last observation with Markov's inequality,
we arrive at
\begin{eqnarray*}
  \mathbb{P} \left(\frac{\mathbf{W}- \bs{\mu}}{\bs{\sigma}} \preceq -\mathbf{t} \right) = \mathbb{P}\left(\bs{\theta}^t \left( \frac{\mathbf{W}-\bs{\mu}}{\bs{\sigma}} \right) \geq \bs{\theta}^t \mathbf{t} \right) \leq \exp \left(-\bs{\theta}^t \mathbf{t} + \frac{K}{\sigma_{(1)}} \left\|\frac{\bs{\mu}}{\bs{\sigma}} \right\|_2 \|\bs{\theta}\|_2^2
    \right).
\end{eqnarray*}

Substituting $\bs{\theta}=
\frac{-\mathbf{t}}{2\frac{K}{\sigma_{(1)}}\left\|\frac{\bs{\mu}}{\bs{\sigma}}
\right\|_2} \preceq \mathbf{0}$, result follows.  \hfill $\square$

\section{Two applications}\label{appssection}

In this section, we will discuss two applications of Theorem
\ref{sbmainthm} which will be on joint distributions of (1) locally
dependent random variables and (2) the number of patterns in
uniformly random permutations.

\subsection{Local dependence}

Now we show that our results above can be used to obtain
concentration bounds for a random vector
$\mathbf{W}=(W_1,W_2,...,W_k)$ with nonnegative coordinates that are
functions of a subset of a collection of independent random
variables. First part of the following lemma was used in
\cite{ghosh2} for univariate concentration results.

\begin{lemma}\label{localdepconst}
Let $\mathcal{V}=\{1,2,...,k\}$ and $\{C_v, v \in \mathcal{V}\}$ be
a collection of independent random variables, and for each $i \in
\mathcal{V}$, let $\mathcal{V}_{i} \subset \mathcal{V}$ and $W_i =
W_i(C_v, v \in \mathcal{V}_i)$ be a nonnegative random variable with
nonzero and finite mean.

\begin{itemize}
  \item[i.] \cite{ghosh2} If $\{C_v^i, v \in \mathcal{V}_i\}$ has distribution
$$dF^i(c_v, v \in \mathcal{V}_i) = \frac{W_i(c_v, v \in
\mathcal{V}_i)}{\mathbb{E}[W_i(C_v, v \in \mathcal{V}_i)]} dF(c_v,
v\in \mathcal{V}_i)$$ and is independent of $\{C_v, v \in
\mathcal{V}\}$, letting $$W_j^i = W_j(C_v^i, v \in \mathcal{V}_j
\cap \mathcal{V}_i, C_u, u  \in \mathcal{V}_j \cap \mathcal{V}_i^c
),$$ the collection $\mathbf{W}^i = \{W_j^i, j \in \mathcal{V}\}$
has the $\mathbf{W}$ size biased distribution in direction $i$.
  \item[ii.] Further if we assume that $W_i \leq M$ for each $i$, then we have
$$\|\mathbf{W}^i - \mathbf{W}\|_2 \leq \sqrt{b} M$$ where $b= \max_{i} |\{j : \mathcal{V}_j \cap \mathcal{V}_i \neq \emptyset\}|.$
\end{itemize}
\end{lemma}

\begin{proof}
Proof of the fact that $\mathbf{W}^i = \{W_j^i, j \in \mathcal{V}\}$
has the $\mathbf{W}$ size biased distribution in direction $i$ can
be found in \cite{ghosh2}. For the second part, we note that by the
construction in the statement, we have $W_j = W_j^i$ whenever
$\mathcal{V}_j \cap \mathcal{V}_i = \emptyset.$ Thus,
$$\|\mathbf{W}^i - \mathbf{W}\|_2 = \left(\sum_{j=1}^k |W_j^i -
W_j|^2 \right)^{1/2} \leq (M^2 \max_{i} |\{j : \mathcal{V}_j \cap
\mathcal{V}_i \neq \emptyset\}|)^{1/2}= \sqrt{b} M.$$
\end{proof}
In conclusion, we note that in the case of local dependence as
described above, we can use Theorem \ref{sbmainthm} to obtain
concentration bounds for $\mathbf{W}=(W_1,W_2,...,W_k)$ with
$K=\sqrt{b}M$. This provides a natural generalization to the
argument given in Section \ref{mainsection} for vectors with
independent coordinates.

Ghosh and Goldstein \cite{ghosh2} also provides two specific
applications of this result on sliding $m$ window statistics and
local extrema on a lattice in a univariate setting. The discussion
above immediately yields multivariate concentration bounds for each
of these problems, but we do not include the details here as they
will be repetitions of the steps done in \cite{ghosh2}.

\subsection{Counting patterns}

Let $\tau_1,\tau_2,...,\tau_k \in S_m$ be $k$ distinct permutations
from $S_m,$ the permutation group on $m\geq 3$ elements. Also let
$\pi$ be a uniformly random permutation in $S_n$, where $n \geq m$
and set $\mathcal{V}=\{1,2,...,n\}$. Denoting $$\mathcal{V}_{s}=\{s,
s+1,...,s+m-1\} \quad \text{for} \; \, s \in \mathcal{V}$$ where
addition of elements of $\mathcal{V}$ is modulo $n$, we say the
pattern $\tau$ appears at location $s \in \mathcal{V}$ if the values
$\{\pi(v)\}_{v \in \mathcal{V}_{s}}$ and $\{\tau(v)\}_{v \in
\mathcal{V}_{1}}$ are in the same relative order. Equivalently, the
pattern $\tau$ appears at $s$ if and only if $\pi(\tau^{-1}(v) + s
-1), v \in \mathcal{V}_1$ is an increasing sequence. Our purpose
here is to prove concentration bounds using Theorem \ref{sbmainthm}
for the multivariate random variable $\mathbf{W}=(W_1,W_2,...,W_k)$
where $W_i$ counts the number of times pattern $\tau_i$ appears in
$\pi$. This problem was previously studied in \cite{ghosh2} for the
univariate case.

For $\tau \in S_m$, let $I_j(\tau)$ be the indicator that
$\tau(1),...,\tau(m-j)$ and $\tau(j+1),...,\tau(m)$ are in the same
relative order. Following the calculations in \cite{ghosh2}, for
$i=1,...,k$, we have

\begin{equation}\label{expinpattern}
    \mu_i = \mathbb{E}[W_i] = \frac{n}{m!}
\end{equation}
and
\begin{equation}\label{varinpattern}
    \sigma_i^2 = Var(W_i) = n \left(\frac{1}{m!} \left(1 -\frac{2m-1}{m!} \right)+2 \sum_{j=1}^{m-1} \frac{I_j(\tau_i)}{(m+j)!} \right)
\end{equation}

Now we are ready to give our main result.

\begin{theorem}
With the setting as above, if $\mathbf{W} = (W_1, W_2,...,W_k)$,
then the conclusions of Theorem \ref{sbmainthm} hold with mean and
variance as in (\ref{expinpattern}) and (\ref{varinpattern}), and
$$K_1 = \frac{2k(2m-1) m!}{m!-2m +2} \quad \text{and} \quad K_2 =
\frac{\sqrt{k}(2m-1) m!}{2 \sqrt{n (m! - 2m +1)}}.$$
\end{theorem}

\begin{proof}

Letting $\pi$ be a  uniformly random permutation in $S_n$, and
$X_{s,\tau}$ the indicator that $\tau$ appears at $s$,
$$X_{s,\tau}(\pi(v), v\in \mathcal{V}_{s}) =
\mathbbm{1}(\pi(\tau^{-1}(1)+s-1) < ... < \pi(\tau^{-1}(m)+s-1)),$$
the sum $W=\sum_{s \in \mathcal{V}} X_{s, \tau}$ counts the number
of $m-$element-long segments of $\pi$ that have the same relative
order as $\tau.$

Now let $\sigma_{s}$ the permutation  in $S_m$ so that
$$\pi(\sigma_{s}(1)+s-1) < ....<\pi(\sigma_{s}(m)+s-1)$$
and set

$$
\pi_1^{s}(v) =
\begin{cases}
\pi(\sigma_s (\tau_1(v- s +1)) + s -1), & \text{if }
v \in \mathcal{V}_{s} \\
\pi(v)  & \text{if }v \notin \mathcal{V}_{s}
\end{cases}
$$
In other words, $\pi_1^{s}$ is the permutation $\pi$ with the values
$\pi(v), v \in \mathcal{V}_{s}$ reordered so that
$\pi_1^{s}(\gamma)$ for $\gamma \in \mathcal{V}_{s}$ are in the same
relative order as $\tau_1.$ Similarly we can define
$\pi_{2}^{s},...,\pi_k^s$ corresponding to $\tau_2,...,\tau_k,$
respectively.

To obtain $\mathbf{W}^i$, the $\mathbf{W}$ size biased variate in
direction $i$ for $i=1,2,...,k$, pick an index $\beta$ uniformly
from $\{1,...,n\}$ and set $W_j^i = \sum_{s \in \mathcal{V}} X_{s,
\tau_j}(\pi_i^{\beta}).$ Then $\mathbf{W}^i=(W_1^i,W_2^i,...,W_k^i)
$ for $i=1,2,...,k.$ The fact that we indeed obtain the desired size
bias variates follows from results in \cite{goldsteinpattern}.

Since $\pi_1^{\beta}, \pi_2^{\beta},...,\pi_k^{\beta}$ agree with
$\pi$ on all the indices leaving out $\mathcal{V}_{\beta}$ and
$|\mathcal{V}_{\beta}|=m,$ we obtain $|W_j^i - W_j| \leq 2m-1$ for
$i,j =1,2,...,k.$ Hence, $\|\mathbf{W}^i-\mathbf{W}\|_2 \leq
\sqrt{k}(2m-1)$ for each $i\in \{1,2,...,k\}.$

Now recall from (\ref{varinpattern}) that $    \sigma_i^2 = n
\left(\frac{1}{m!} \left(1 -\frac{2m-1}{m!} \right)+2
\sum_{j=1}^{m-1} \frac{I_j(\tau_i)}{(m+j)!} \right)$ for
$i=1,2,...,k.$ Since $0 \leq I_j \leq 1$, one can obtain a variance
lower bound by setting $I_k=0$. In particular, this yields
$$\sigma_{(1)}^2 \geq \frac{n}{m!} \left(1-\frac{2m-1}{m!}
\right).$$ Since  the constants $K_1$ and $K_2$ in Theorem
\ref{sbmainthm} can be replaced by larger constants, result follows
from simple computations.
\end{proof}

\end{document}